\documentclass[12pt, reqno]{amsart}
\usepackage{amsmath, amsthm, amscd, amsfonts, amssymb, graphicx, color}
\usepackage[english]{babel}
\usepackage{cite}
\usepackage[bookmarksnumbered, colorlinks, plainpages]{hyperref}

\newtheorem{theorem}{Theorem}[section]
\newtheorem{lemma}[theorem]{Lemma}

\newtheorem{corollary}[theorem]{Corollary}
\theoremstyle{definition}

\theoremstyle{remark}

\numberwithin{equation}{section}
\begin{document}

\title[2-Local   derivations  on  algebras]{2-Local   derivations
 on  algebras of locally measurable operators}


\author{Sh. A. Ayupov}
\address{Institute of
 Mathematics National University of
Uzbekistan,
 100125  Tashkent,   Uzbekistan
 and
 the Abdus Salam International Centre
 for Theoretical Physics (ICTP),
  Trieste, Italy}
\email{sh$_{-}$ayupov@mail.ru}
 \thanks{The second  named author would like
to acknowledge the hospitality of the Chern Institute of
Mathematics, Nankai University  (China). }

\author{K. K. Kudaybergenov}
\address{Department of Mathematics, Karakalpak state university\\
Ch. Abdirov 1,  230113, Nukus,    Uzbekistan}
\email{karim2006@mail.ru}

\author{A. K. Alauadinov}
\address{Institute of
 Mathematics National University of
Uzbekistan,
 100125  Tashkent,   Uzbekistan}
\email{amir$_{-}$t85@mail.ru}





\begin{abstract}
The paper is devoted to $2$-local derivations  on the algebra
$LS(M)$ of all locally measurable   operators
affiliated with a type I$_\infty$ von Neumann algebra $M.$
We prove that every  $2$-local derivation on
$LS(M)$ is a derivation.
\end{abstract}

\maketitle
\section{Introduction}

Given an algebra $\mathcal{A},$ a linear operator $D:\mathcal{A}\rightarrow \mathcal{A}$ is
called a \textit{derivation}, if $D(xy)=D(x)y+xD(y)$ for all $x,
y\in \mathcal{A}$ (the Leibniz rule). Each element $a\in \mathcal{A}$ implements a
derivation $D_a$ on $\mathcal{A}$ defined as $D_a(x)=[a, x]=ax-xa,$ $x\in \mathcal{A}.$ Such
derivations $D_a$ are said to be \textit{inner derivations}. If
the element $a,$ implementing the derivation $D_a,$ belongs to a
larger algebra $\mathcal{B}$ containing $\mathcal{A},$ then $D_a$ is called \textit{a
spatial derivation} on $\mathcal{A}.$

There exist various types of linear operators which are close to derivations
\cite{Kad,  Kim,  Sem1}. In particular
R.~Kadison \cite{Kad} has introduced and
investigated so-called local derivations on von Neumann algebras and some polynomial algebras.

A linear operator $\Delta$ on an algebra $\mathcal{A}$ is called a
\textit{local derivation} if given any $x\in \mathcal{A}$ there exists a
derivation $D$ (depending on $x$) such that $\Delta(x)=D(x).$  The
main problems concerning this notion are to find conditions under
which local derivations become derivations and to present examples of algebras
with local derivations that are not derivations \cite{Kad}. In particular
Kadison \cite{Kad} has proved that each
continuous local derivation from a von Neumann algebra
$M$ into a dual $M$-bimodule is a derivation.

In 1997, P. Semrl \cite{Sem1}  introduced the concept of
$2$-local derivations and automorphisms.
A  map $\Delta:\mathcal{A}\rightarrow\mathcal{A}$  (not linear in general) is called a
 $2$-\emph{local derivation} if  for every $x, y\in \mathcal{A},$  there exists
 a derivation $D_{x, y}:\mathcal{A}\rightarrow\mathcal{A}$
such that $\Delta(x)=D_{x, y}(x)$  and $\Delta(y)=D_{x, y}(y).$
A  map $\Theta:\mathcal{A}\rightarrow\mathcal{A}$  (not linear in general) is called a
 $2$-\emph{local automorphism}
 if  for every $x, y\in \mathcal{A},$  there exists
 an automorphism  $\Phi_{x, y}:\mathcal{A}\rightarrow\mathcal{A}$
such that $\Theta(x)=\Phi_{x, y}(x)$  and $\Theta(y)=\Phi_{x, y}(y).$
Local and $2$-local maps have been studied on different operator algebras by many
authors
\cite{Nur,  AKN, JMAA,  Bre1, AKNA, Liu,  Kad, Kim, Lar, Lin, Mol, Sem1, Sem2}.

In \cite{Sem1}, P. Semrl described
$2$-local derivations and automorphisms on the algebra $B(H)$ of
 all bounded linear operators on the infinite-dimensional
separable Hilbert space $H.$  A similar
description for the finite-dimensional case appeared later in \cite{Kim},
\cite{Mol}. In the paper \cite{Lin}
$2$-local derivations and automorphisms have been described on
matrix algebras over finite-dimensional division rings.
J.~H.~Zhang and  H.~X.~Li  \cite{Zhang}
 described $2$-local  derivations on symmetric
 digraph algebras  and constructed  a $2$-local derivation
  on  the algebra
of all upper triangular complex $2\times 2$-matrices
which is not a
derivation.
In \cite{JMAA}   first two authors considered
$2$-local derivations on the algebra $B(H)$ of all linear bounded operators on an
arbitrary (no separability is assumed) Hilbert space $H$ and proved that every $2$-local
derivation on $B(H)$ is a derivation.

The present paper  is devoted to study
$2$-local derivations on   $\ast$-subalgebras
of the algebra  $LS(M)$
 of all locally measurable operators with
respect to type I$_\infty$ von Neumann algebra $M.$
 We prove that every $2$-local derivations on every  $\ast$-subalgebra $\mathcal{A}$
 in  $LS(M)$,
such that $M\subseteq\mathcal{A}$, is a  derivation.

\section{Algebra of locally measurable operators}

Let  $B(H)$ be the $\ast$-algebra of all
bounded linear operators on a Hilbert space $H,$
and let $\textbf{1}$ be the identity operator
on $H.$ Consider a von Neumann algebra $M\subset B(H).$
 Denote by $P(M)=\{p\in M: p=p^2=p^\ast\}$ the lattice of all projections in
$M$ and by $P_{fin}(M)$
the set  of all finite
projections in $P(M).$

A linear subspace  $\mathcal{D}$ in  $H$ is said to be
\emph{affiliated} with  $M$ (denoted as  $\mathcal{D}\eta M$), if
$u(\mathcal{D})\subset \mathcal{D}$ for every unitary  $u$ from
the commutant
$$M'=\{y\in B(H):xy=yx, \,\forall x\in M\}$$ of the von Neumann algebra $M.$

A linear operator  $x: \mathcal{D}(x)\rightarrow H,$ where the  domain  $\mathcal{D}(x)$
of $x$ is a linear subspace of $H,$ is said to be \textit{affiliated} with  $M$ (denoted as  $x\eta M$) if
$\mathcal{D}(x)\eta M$ and $u(x(\xi))=x(u(\xi))$
 for all  $\xi\in
\mathcal{D}(x)$  and for every unitary  $u\in M'.$

A linear subspace $\mathcal{D}$ in $H$ is said to be \textit{strongly
dense} in  $H$ with respect to the von Neumann algebra  $M,$ if
\begin{itemize}
\item  $\mathcal{D}\eta M;$

\item  there exists a sequence of projections
$\{p_n\}_{n=1}^{\infty}$ in $P(M)$  such that
$p_n\uparrow\textbf{1},$ $p_n(H)\subset \mathcal{D}$ and
$p^{\perp}_n=\textbf{1}-p_n$ is finite in  $M$ for all
$n\in\mathbf{N}$.
\end{itemize}
A closed linear operator  $x$ acting in the Hilbert space $H$ is said to be
\textit{measurable} with respect to the von Neumann algebra  $M,$ if
 $x\eta M$ and $\mathcal{D}(x)$ is strongly dense in  $H.$

 Denote by $S(M)$  the set of all linear operators on $H,$ measurable with
respect to the von Neumann algebra $M.$ If $x\in S(M),$
$\lambda\in\mathbf{C},$ where $\mathbf{C}$  is the field of
complex numbers, then $\lambda x\in S(M)$  and the operator
$x^\ast,$  adjoint to $x,$  is also measurable with respect to $M$
(see \cite{Seg}). Moreover, if $x, y \in S(M),$ then the operators
$x+y$  and $xy$  are defined on dense subspaces and admit closures
that are called, correspondingly, the strong sum and the strong
product of the operators $x$  and $y,$  and are denoted by
$x\stackrel{.}+y$ and $x \ast y.$ It was shown in  \cite{Seg} that
$x\stackrel{.}+y$ and $x \ast y$ belong to $S(M)$ and
these algebraic operations make $S(M)$ a $\ast$-algebra with the
identity $\textbf{1}$  over the field $\mathbf{C}.$ It is clear that, $M$ is a
$\ast$-subalgebra of $S(M).$ In what follows, the strong sum and
the strong product of operators $x$ and $y$  will be denoted in
the same way as the usual operations, by $x+y$  and $x y.$

A closed linear operator $x$ in  $H$  is said to be \emph{locally
measurable} with respect to the von Neumann algebra $M,$ if $x\eta
M$ and there exists a sequence $\{z_n\}_{n=1}^{\infty}$ of central
projections in $M$ such that $z_n\uparrow\textbf{1}$ and $z_nx \in
S(M)$ for all $n\in\mathbf{N}$ (see \cite{Yea}).

Denote by $LS(M)$ the set of all linear operators that are locally
measurable with respect to $M.$ It was proved in \cite{Yea}  that
$LS(M)$ is a $\ast$-algebra over the field $\mathbf{C}$ with the
identity $\textbf{1},$ the operations of strong addition, strong
multiplication, and passing to the adjoint. In such a case, $S(M)$
is a $\ast$-subalgebra in $LS(M).$ In the case where $M$ is a
finite von Neumann algebra or a factor, the algebras $S(M)$ and
$LS(M)$ coincide. This is not true in the general case. In
\cite{Mur2} the class of von Neumann algebras $M$ has been described for
which the algebras  $LS(M)$ and  $S(M)$ coincide.

We
say that a measure $\mu$  on a measure space
$(\Omega,\Sigma,\mu)$
 has the direct sum property if there is a family
 $\{\Omega_{i}\}_{i\in
J}\subset\Sigma,$ $0<\mu(\Omega_{i})<\infty,\,i\in J,$ such that
for any $A\in\Sigma,\,\mu(A)<\infty,$ there exist a countable
subset $J_{0 }\subset J$ and a set  $B$ with zero measure such
that $A=\bigcup\limits_{i\in J_{0}}(A\cap
\Omega_{i})\cup B.$

It is well-known (see e.g.~\cite{Seg}) that for each  commutative von
Neumann algebra
 $M$ there exists a  measure
space $(\Omega, \Sigma, \mu)$ with $\mu$  having the direct sum property
such that $M$
is $\ast$-isomorphic to the algebra $L^{\infty}(\Omega, \Sigma, \mu)$
of all (equivalence classes of) complex essentially bounded
measurable functions on  $(\Omega, \Sigma, \mu)$ and in this case
$LS(M)=S(M)\cong L^{0}(\Omega, \Sigma, \mu),$ where $L^{0}(\Omega,
\Sigma, \mu)$ the algebra of all (equivalence classes of) complex
measurable functions on $(\Omega, \Sigma, \mu).$

Further we consider the algebra  $S(Z(M))$  of operators which are measurable
with respect to the  center $Z(M)$ of the von Neumann algebra $M.$
Since  $Z(M)$ is an abelian von Neumann algebra  it
 is $\ast$-isomorphic to $ L^{\infty}(\Omega, \Sigma, \mu)$
   for an appropriate measure space $(\Omega, \Sigma, \mu)$.
   Therefore the algebra  $S(Z(M))$ coincides
   with $Z(LS(M))$ and  can be
 identified with the algebra $ L^{0}(\Omega, \Sigma, \mu).$

Let  $M$ be a von Neumann algebra. Given an element   $x\in LS(M)$
the smallest   projection  $p$ in $M$ with
$xp=x$  is called the  \emph{right support}
of   $x$ and denoted by $r(x).$
The   \emph{left support} $l(x)$ is smallest projection
$p$ in $M$ with $p x=x.$  For a $\ast$-subalgebra  $\mathcal{A}\subset LS(M)$ denote
$$\mathcal{F}(\mathcal{A})=\{x\in A: l(x)\in P_{fin}(M)\}.$$

From the definition  of the algebra $\mathcal{F}(\mathcal{A})$ we have that the
following properties are equivalent:
\begin{enumerate}
\item $x\in \mathcal{F}(\mathcal{A});$

\item $\exists\, p\in P_{fin}(M)$ such that  $px=x;$

\item $\exists\, p\in P_{fin}(M)$ such that $xp=x;$

\item  $\exists\, p\in P_{fin}(M)$ such that  $pxp=x.$
\end{enumerate}
Note that  $\mathcal{F}(A)$ is an $\ast$-ideal in  $\mathcal{A}.$
Moreover the algebra $\mathcal{F}(\mathcal{A})$
 is semi-prime, i.e. if  $a\in \mathcal{F}(\mathcal{A})$ and
$a \mathcal{F}(\mathcal{A}) a=\{0\}$
then $a=0.$ Indeed, let $a\in \mathcal{F}(\mathcal{A})$
 and $a \mathcal{F}(\mathcal{A}) a=\{0\},$
i.e. $axa=0$ for all $x\in \mathcal{F}(\mathcal{A}).$ In particular for
$x=a^{\ast}$ we have $aa^{\ast}a=0$ and hence $a^{\ast}aa^{\ast}a=0,$ i.e.
$|a|^4=0.$ Therefore $a=0.$

Recall the definition of the faithful normal semifinite extended
center valued trace on the algebra $M$ (see \cite{Tak}).

 Let  $M$ be an arbitrary von Neumann algebra with the
 center \linebreak $Z(M)\equiv L^\infty(\Omega,\Sigma, \mu).$
 By  $L_+$ we denote the set of all measurable  functions  $f:
(\Omega,\Sigma, \mu)\rightarrow [0,{\infty}]$ (modulo functions
equal to zero $\mu$-almost everywhere).
Then there exists a map  $\Phi: M_+\rightarrow L_{+}$ with the
following properties:

\begin{enumerate}
\item $\Phi(x+y)=\Phi(x)+\Phi(y)$  for $x, y \in M_+;$

\item $\Phi(ax)=a\Phi(x)$  for $a\in Z(M)_+, x \in M_+;$

\item $\Phi(xx^\ast)=\Phi(x^\ast x);$

\item $\Phi(x^\ast x)=0$ $\Rightarrow$ $x=0;$

\item  $\Phi\left(\sup \limits_{i \in
J}x_{i}\right)=\sup \limits_{i \in
J}\Phi(x_{i})$
for any bounded increasing net $\{x_i\}$ in $M_+.$
\end{enumerate}

This map  $\Phi: M_+\rightarrow L_+,$ is a called the \emph{extended
center valued trace} on  $M.$

The set $\{x\in M: \Phi(x^\ast x)\in Z(M)\}$ is an ideal $M.$ If this ideal
is $\sigma$-weakly dense in $M,$ then
$\Phi$ is said to be \textit{semifinite}.

It is well-known (see e.g.~\cite{Tak}) that
a  von
Neumann algebra
 $M$ is semifinite
 if and only if $M$ admits a faithful,
 semifinite,
 normal extended center valued trace.

Let us remark that a projection
$p\in M$ is finite if and only if $\Phi(p)\in S(Z(M)).$
Hence for any $x\in \mathcal{F}(LS(M))\cap M_+$ we have that
$\Phi(x)\in S(Z(M)).$

 Note that the algebra $LS(M)$ has the following remarkable property:
 given any family  $\{z_i\}_{i\in I}$ of mutually orthogonal
central projections in $M$ with $\bigvee\limits_{i\in
I}z_i=\textbf{1}$ and a  family of elements $\{x_i\}_{i\in I}$ in
$LS(M)$ there exists a unique element $x\in LS(M)$ such that $z_i
x=z_i x_i$ for all $i\in I.$ This element is denoted by
$x=\sum\limits_{i\in I}z_i x_i$
 (see \cite{Mur2}). Conversely if
$M$ is  a  type I von Neumann algebra then for an arbitrary
element $x\in LS(M)$  there exists a sequence $\{z_n\}$ of
mutually orthogonal central projections with $\bigvee\limits_{n\in
\mathbb{N}}z_n=\mathbf{1}$ such that $z_n x\in M$ for all $n\in
\mathbb{N}$ (see \cite{Alb2}). For $0\leq x\in \mathcal{F}(LS(M))$
set
\begin{equation}\label{TR}
\Phi(x)=\sum\limits_{n\in \mathbb{N}}z_n \Phi(z_n x).
\end{equation}
Since the trace $\Phi$ is $Z(M)$-homogeneous, the equality
\eqref{TR} gives a well-defined map
from
$\mathcal{F}(LS(M))_+$ into $S(Z(M)).$

Since each  element of $\mathcal{F}(LS(M))$  is a finite linear combinations
 of positive  elements from $\mathcal{F}(LS(M))$
we can naturally extend $\Phi$ to a $S(Z(M))$-valued trace  on  $\mathcal{F}(LS(M)).$

Now let $\mu$ be an arbitrary faithful normal semifinite trace on $Z(M).$ Put
$\tau=\mu\circ\Phi.$
Then by \cite[Lemma 2.16]{Tak} we have that
$$
\tau(xy)=\tau(yx)
$$
for all $x\in M,$ $y\in \mathcal{F}(LS(M))\cap M.$
Therefore
$$
\Phi(xy)=\Phi(yx)
$$
for all $x\in LS(M),$ $y\in \mathcal{F}(LS(M)).$
Since  the trace $\Phi$ maps the set
$\mathcal{F}(LS(M))$ into $S(Z(M))$ and  $\mathcal{F}(LS(M))$ is an ideal in
$LS(M)$  we have
$$
\Phi(axy) =  \Phi((ax)y) =
\Phi((ya)x) =
\Phi(xya),
$$
i.e.
\begin{equation}\label{tr}
\Phi(axy) =\Phi(xya)
\end{equation}
for all $a, x\in LS(M),$ $y\in \mathcal{F}(LS(M)).$

\section{Main results}

Let $D$ be a derivation on $LS(M).$ Then
 $D$  maps the ideal
$\mathcal{F}(LS(M))$ into itself. Indeed, for any
$x\in \mathcal{F}(LS(M))$ there exists a finite projection $p\in M$
such that $x=xp.$ Then
$$
D(x)=D(xp)=D(x)p+xD(p),
$$
and therefore $D(x)\in \mathcal{F}(LS(M)).$
Hence any $2$-local derivation on $LS(M)$ also maps
$\mathcal{F}(LS(M))$ into itself.

\begin{lemma}\label{A} Let
 $b\in LS(M)$ be an arbitrary element.
If  $\Phi(x b)=0$ for all $x\in \mathcal{F}(LS(M))$
then  $b=0.$
\end{lemma}

\begin{proof} Let  $b\in LS(M).$ For any
finite  projection  $e\in LS(M)$ we have $e b^\ast\in
\mathcal{F}(LS(M))$ and therefore
by the assumption of the lemma it follows that
$\Phi(e b^\ast b)=0.$
Thus
$$
0=\Phi(e b^\ast b)=\Phi(e^2 b^\ast b)=
\Phi(e b^\ast b e)=\Phi((be)^\ast(be)),
$$
i.e.
$$
\Phi((be)^\ast(be))=0.
$$
Since the trace $\Phi$ is faithful, we obtain $(be)^\ast(be)=0,$ i.e. $be=0.$

Now  take a family  of finite    projections
$\{e_\alpha\}_{\alpha\in J}$ in
 $M$ such that $e_\alpha\uparrow\mathbf{1}.$
Then
$$
0=be_\alpha b^\ast\uparrow b b^\ast,
$$
i.e. $bb^\ast=0.$ Thus $b=0.$
The proof is complete.
\end{proof}

\begin{lemma}\label{H}
Let $M$ be an arbitrary von Neumann algebra of type I$_\infty$
and let    $\Delta: LS(M)\rightarrow LS(M)$ be a $2$-local derivation.
Then

\begin{enumerate}
\item  $\Delta$ is $S(Z(M))$-homogenous,  i.e.
$\Delta(cx)=c\Delta(x)$ for all $c\in S(Z(M)),$
$x\in LS(M);$

\item $\Delta(x^2)=\Delta(x)x+x\Delta(x)$ for all
$x\in LS(M).$
\end{enumerate}
 \end{lemma}

\begin{proof} (1) For
 each $x\in LS(M),$ and for
 $c\in S(Z(M))$ there exists a derivation $D_{x, cx}$
 such that $\Delta(x)=D_{x, c x}(x)$ and
 $\Delta(c x)=D_{x, c x}(c x).$
Since $M$ is a  type I$_\infty$ then
by \cite[Theorem 2.7]{Alb2} every derivation on $LS(M)$ is inner,
in particular,
$S(Z(M))$-linear. Therefore
$$
\Delta(c x)=D_{x, c x}(c x)=c
D_{x, c x}(x)=c\Delta(x).
$$
Hence, $\Delta$ is $S(Z(M))$-homogenous.

(2)
For each $x\in LS(M),$ there exists a derivation $D_{x, x^2}$
 such that $\Delta(x)=D_{x, x^2}(x)$ and $\Delta(x^2)=D_{x, x^2}(x^2).$ Then
$$
\Delta(x^2)=D_{x, x^2}(x^2)=D_{x, x^2}(x)x+xD_{x, x^2}(x)=\Delta(x)x+x\Delta(x)
$$
 for all $x\in LS(M).$
 The proof is complete.
\end{proof}

\begin{lemma}\label{B}
Let $M$ be an arbitrary von Neumann algebra of type I$_\infty.$
If   $\Delta: LS(M)\rightarrow LS(M)$ is a $2$-local derivation such that
 $\Delta|_{\mathcal{F}(LS(M))}\equiv 0,$ then $\Delta\equiv 0.$
 \end{lemma}

\begin{proof}
Let $\Delta: LS(M)\rightarrow LS(M)$ be a $2$-local derivation such that
 $\Delta|_{\mathcal{F}(LS(M))}\equiv 0.$
For arbitrary $x\in LS(M)$ and $y\in \mathcal{F}(LS(M))$
there exists a derivation $D_{x, y}$ on $LS(M)$
 such that
$\Delta(x)=D_{x, y}(x)$  and $\Delta(y)=D_{x, y}(y).$
By \cite[Theorem 2.7]{Alb2} there exists element $a\in LS(M)$
such that
$$
[a, xy]=D_{x, y}(xy)=D_{x, y}(x)y+xD_{x, y}(y)=\Delta(x)y+x\Delta(y),
$$
i.e.
$$
[a, xy]=\Delta(x)y+x\Delta(y).
$$
Since $y\in \mathcal{F}(LS(M))$ we have $\Delta(y)=0,$ and therefore
$
[a, xy]=\Delta(x)y.
$
By the equality \eqref{tr} we obtain that
$$
0 = \Phi(axy-xya)=\Phi\left([a, xy]\right)=\Phi\left(\Delta(x)y\right),
$$
i.e.
$\Phi(\Delta(x)y)=0$ for all $y\in \mathcal{F}(LS(M)).$
By Lemma \ref{A}
we have that $\Delta(x)=0.$
The proof is complete. \end{proof}

\begin{lemma}\label{AD}
Let $M$ be an arbitrary von Neumann algebra of type I$_\infty$
and let    $\Delta: LS(M)\rightarrow LS(M)$ is a $2$-local derivation.
Then  the restriction $\Delta|_{\mathcal{F}(LS(M))}$ of the operator $\Delta$
on $\mathcal{F}(LS(M))$ is  additive.
 \end{lemma}

\begin{proof} Let $\Delta: LS(M)\rightarrow LS(M)$ be
 a $2$-local derivation.
For each $x, y\in \mathcal{F}(LS(M))$ there exists a derivation $D_{x, y}$ on $LS(M)$
 such that
$\Delta(x)=D_{x, y}(x)$  and $\Delta(y)=D_{x, y}(y).$
By \cite[Theorem 2.7]{Alb2} there exists an  element $a\in LS(M)$
such that
$$
[a, xy]=D_{x, y}(xy)=D_{x, y}(x)y+xD_{x, y}(y)=\Delta(x)y+x\Delta(y),
$$
i.e.
$$
[a, xy]=\Delta(x)y+x\Delta(y).
$$
Similarly as in Lemma \ref{B}  we have
$$
0 = \Phi(axy-xya)=
\Phi([a, xy])=\Phi\left(\Delta(x)y+x\Delta(y)\right),
$$
i.e.
$\Phi(\Delta(x)y)=-\Phi(x \Delta(y)).$
For arbitrary $u, v, w\in \mathcal{F}(LS(M)),$
 set $x=u+v,$ $y=w.$ Then from above we obtain
$$
\Phi(\Delta(u+v)w)=-\Phi((u+v)\Delta(w))=
$$
$$
=-\Phi(u\Delta(w))- \Phi(v\Delta(w))=
\Phi(\Delta(u)w)+\Phi(\Delta(v)w)=\Phi((\Delta(u)+\Delta(v))w),
$$
and so
$$
\Phi((\Delta(u+v)-\Delta(u)-\Delta(v))w)=0
$$ for all $u, v, w\in \mathcal{F}(LS(M)).$
Denote
$
b=\Delta(u+v)-\Delta(u)-\Delta(v)
$ and put
$
w=b^\ast.$ Then
$\Phi(bb^\ast)=0.$
Since the trace $\Phi$ is faithful it follows that  $bb^\ast=0,$ i.e. $b=0.$
Therefore
$$
\Delta(u+v)=\Delta(u)+\Delta(v),
$$
i.e. $\Delta$ is an additive map on $\mathcal{F}(LS(M)).$
The proof is complete.
\end{proof}

The following theorem is the main result of this paper.

\begin{theorem}\label{Main}
Let $M$ be an arbitrary von Neumann algebra of type I$_\infty$
and let $\mathcal{A}$ be a $\ast$-subalgebra
of $LS(M)$ such that $M\subseteq \mathcal{A}.$
 Then every $2$-local
derivation $\Delta: \mathcal{A}\rightarrow \mathcal{A}$
 is a
derivation.
\end{theorem}

\begin{proof}
First we consider the case $\mathcal{A}=LS(M).$
By Lemma \ref{AD} the restriction $\Delta|_{\mathcal{F}(LS(M))}$ of the operator $\Delta$
on $\mathcal{F}(LS(M))$ is  additive.
Further by Lemma \ref{H} $\Delta$ is a homogeneous.
 Therefore, the map  $\Delta|_{\mathcal{F}(LS(M))}$  is a linear Jordan derivation on
$\mathcal{F}(LS(M))$ in the sense of \cite{Bre2}. In
\cite[Theorem 1]{Bre2} it is proved that any Jordan derivation on
a semi-prime algebra is a derivation. Since $\mathcal{F}(LS(M))$ is semiprime,
therefore the linear operator $\Delta|_{\mathcal{F}(LS(M))}$ is a derivation on
$\mathcal{F}(LS(M)).$

Since by Lemma \ref{H} $\Delta$ is $S(Z(M))$-homogeneous then
 by  \cite[Corollary 3]{AKN}
the derivation  $\Delta|_{\mathcal{F}(LS(M))} : \mathcal{F}(LS(M))\rightarrow \mathcal{F}(LS(M))$
 is spatial, i.e.
\begin{equation}\label{F}
\Delta(x)=ax-xa, \, x\in \mathcal{F}(LS(M))
\end{equation}
for an appropriate $a\in LS(M).$

Let us show that  $\Delta(x)=ax-xa$ for all  $x\in LS(M).$
Consider the $2$-local derivation  $\Delta_0=\Delta-D_a.$
 Then from  the equality (\ref{F}) we obtain that  $\Delta_0|_{\mathcal{F}(LS(M))}\equiv 0.$
Now by Lemma~\ref{B} it follows that $\Delta_0\equiv 0.$
 This means that  $\Delta=D_a.$

Now let
$\mathcal{A}$ be an arbitrary  $\ast$-subalgebra
of $LS(M)$ such that $M\subseteq \mathcal{A}.$
Since   $M$ is a  type I von Neumann algebra for
any  element
$x\in LS(M)$  there exists  a sequence $\{z_n\}$ of
mutually orthogonal central projections with
$\bigvee\limits_{n\in \mathbb{N}}z_n=\mathbf{1}$
such that $z_n x\in M$ for all $n\in \mathbb{N}.$
Set
\begin{equation}\label{loc}
\tilde{\Delta}(x)=\sum\limits_{n\in \mathbb{N}}z_n \Delta(z_n x).
\end{equation}
Since the map  $\Delta$ is $Z(M)$-homogeneous, the equality
\eqref{loc} gives a well-defined $2$-local derivation on $LS(M).$
From above we have that $\tilde{\Delta}$  is a derivation.
Therefore $\Delta$ is a derivation. The proof is complete.
\end{proof}

\begin{corollary}\label{LSM}
Let $M$ be an arbitrary von Neumann algebra of type I$_\infty.$
 Then every $2$-local
derivation $\Delta: LS(M)\rightarrow LS(M)$
 is a
derivation.
\end{corollary}

\end{document}